\documentclass{tran-l}

\usepackage{amssymb,amsmath,amsthm}
\usepackage{mathrsfs,dsfont}

\newtheorem{theorem}{Theorem}[section]
\newtheorem{lemma}[theorem]{Lemma}
\newtheorem{prop}[theorem]{Proposition}

\theoremstyle{definition}
\newtheorem{example}[theorem]{Example}

\theoremstyle{remark}
\newtheorem{remark}[theorem]{Remark}
\numberwithin{equation}{section}

\newcommand{\del}{\partial}

\newcommand{\calF}{\mathscr{F}}

\newcommand{\frakm}{{\mathfrak{m}}}
\newcommand{\frakp}{{\mathfrak{p}}}

\newcommand{\CC}{\mathds{C}}
\newcommand{\KK}{\mathds{K}}
\newcommand{\NN}{\mathds{N}}

\newcommand{\QQ}{\mathds{Q}}
\newcommand{\RR}{\mathds{R}}
\newcommand{\ZZ}{\mathds{Z}}
\newcommand{\II}{\mathds{I}}

\DeclareMathOperator{\conv}{Conv}

\DeclareMathOperator{\gr}{gr}

\DeclareMathOperator{\Hom}{Hom}

\DeclareMathOperator{\Spec}{Spec}
\DeclareMathOperator{\Var}{Var}
\DeclareMathOperator{\Hole}{Hole}
\DeclareMathOperator{\Ch}{Ch}
\DeclareMathOperator{\cycle}{CC}
\DeclareMathOperator{\ann}{ann}

\begin{document}

\title{$D$-module structure of Local Cohomology modules of toric algebras}

\author{Jen-Chieh Hsiao}
\address{Department of Mathematics, Purdue University, 150 N. University Street, West Lafayette, IN~47907, USA} 
\email{jhsiao@math.purdue.edu}
\thanks{The author was partially supported by NSF under grant DMS~0555319 and DMS~0901123.}


\begin{abstract}
 Let $S$ be a toric algebra over a field $\KK$ of characteristic $0$ and let $I$ be a monomial ideal of $S$. We show that the local cohomology 
 modules $H^i_I(S)$ are of finite length over the ring of differential operators $D(S;\KK)$, generalizing the classical case of a polynomial algebra $S$.
 As an application, we compute the characteristic cycles of some local cohomology modules. 
\end{abstract}

\maketitle

\section{Introduction}

 Lyubeznik \cite{Lyubeznik} introduced an approach of studying local cohomology modules using the theory of $D$-modules. 
 He obtained many finiteness properties of the local cohomology modules $H_I^i(R)$ when $R$ is a regular ring containing a field of characteristic $0$. 
 For example, taking advantage of holonomicity of $H^i_I(R)$ as a $D$-module, he showed that for any maximal ideal $\frakm$ the number 
 of associated prime ideals  of $H^i_I(R)$  contained in $\frakm$ is finite and that all Bass numbers of $H^i_I(R)$ are finite. 
 When $R$ is a regular local ring of positive characteristic, analogous results were obtained by Huneke and Sharp using the Frobenius functor \cite{HS}. 
 When $R$ is not regular, the situation is more subtle. There are characteristic-free examples where $H^i_I(R)$ have infinitely many associated primes \cite{SS}.
 Also, an example by Hartshorne \cite{Ha} shows that in general the Bass numbers can be infinite (see Example~~\ref{hartshorne}).
 
 After \cite{Lyubeznik}, there have been several studies on the finiteness properties of $R_x$, and $H^i_I(R)$ as D-modules for a regular ring $R$, among them \cite{Bog95},
  \cite{Bog02}, \cite{Lyu97}, \cite{ABL}.
 The first $D$-finiteness result of $R_x$ for a singular ring $R$ is due to Takagi and Takahashi \cite{TT} which says $R_x$ is generated by $x^{-1}$ over $D$ 
 when $R$ is a Noetherian graded ring with finite F-representation type. In particular, their theorem applies to the case where $R$ is a normal toric algebra over a perfect field of
 positive characteristic.
 
 In the present article, we study finiteness properties of the localizations $S_f$ and the local cohomology modules $H^i_I(S)$ as $D$-modules 
 where $S$ is a toric algebra (not necessarily normal) over a field $\KK$ of characteristic $0$, $f$ is a monomial element and $I$ is a monomial ideal of $S$. 
 In this case, the ring of differential operators $D(S):=D(S;\KK)$ is much more complicated than the case where $S$ is regular. 
 Using the natural grading of $D(S)$ introduced by Jones \cite{Jones} and Musson \cite{Musson95}, Saito and Traves \cite{Saito-Traves01,Saito-Traves04} 
 gave a detailed description of $D(S)$. Based on their results, we prove that any localization $S_f=S[f^{-1}]$ of $S$ is 
 generated by $f^{-1}$ over $D(S)$ (Theorem~\ref{cyclic}). 
 This implies immediately that $H^i_I(S)$ are $D(S)$-finitely generated if $D(S)$ is left Noetherian. Unfortunately, $D(S)$ is not always left Noetherian \cite{Saito-Takahashi09}. 
 Nonetheless, we can show that $H_I^i(S)$ is actually of finite length as a $D(S)$-module (Theorem~\ref{finite length}). In view of Hartshorne's example (Example~\ref{hartshorne}), 
 this result is quite surprising.
 
 As an application, we compute the characteristic cycles of some local cohomology modules $H^i_I(S)$. Characteristic cycles are 
 formal sums of subvarieties (counted with multiplicities) of the characteristic variety of a $D$-module $M$. Here, the characteristic variety $\Ch (M)$
 is the support of the associated
 graded module $\gr M$ in the spectrum $\Spec(\gr D(S)$) of the associated graded ring $\gr D(S)$.
 When $S$ is a polynomial algebra, one can explicitly compute the Bass numbers and the associated primes of $H_I^i(S)$ from its
 characteristic cycles \cite{Montaner2}. The cohomological dimension of $I$ and the Lyubeznik numbers can also be computed from them \cite{Montaner1}. 
 Through our finiteness results of $H^i_I(S)$, we are able to compute the characteristic cycles of some local cohomology modules.
 We will show that for normal toric algebras $S$ (in fact, for a more general class of toric algebras) the characteristic variety $\Ch (H^{{\dim} S}_{\frak m}(S))$ of the 
 top local cohomology with maximal support is abstractly isomorphic to the ambient toric variety $\Spec (S)$  (Theorem~\ref{chmax}).
 
 In section~\ref{pre}, we briefly recall the notions of local cohomology, toric algebras and the ring of differential operators of a commutative algebra over a field. 
 We also describe the structure of rings of differential operators over toric algebras following the notations in \cite{Saito-Traves01} and \cite{Saito-Traves04}. 
 In section~\ref{finiteness}, our main results on the finiteness properties mentioned above are presented. Also, we relate our finiteness results to the notion of sector partition introduced in \cite{Schafer-Schenzel90} and \cite{Matusevich-Miller}. 
 Some discussions on $\gr D(S)$ and the computations of characteristic cycles are in section~\ref{charcycle}. As suggested by the referee, some relations between
 our results in section \ref{charcycle} and the recent work of Saito \cite{S10} are discussed (see Remarks~\ref{ref1}, \ref{ref2}).\\
 \vskip 1pt
 \textbf{Acknowledgment.}
 This work originates from a conversation between Uli Walther and William Traves.
 The author is grateful to his advisor Uli Walther for introducing this problem and providing many useful discussions and comments.
 He would also like to thank Ezra Miller for the comments on sector partition and Example~\ref{hartshorne}, and Christine Berkesch for carefully reading this paper.
 Special thanks go to the referee of this paper who helped improve the presentation, simplified the proof of Theorem~\ref{cyclic}, 
 and pointed out the relations between our results and the results in \cite{S10}.

\section{Preliminaries} \label{pre}

\subsection{Local Cohomology}

General facts regarding local cohomology can be found in \cite{24h} or \cite{BS}. Here, we only recall some basics.

Let $R$ be a Noetherian commutative ring, $M$ an $R$-module, and $I$ an ideal of $R$. Define $\Gamma_I(M) := \varinjlim \Hom_R(R/I^k,M)$. Then $\Gamma_I$ is a left exact 
$R$-linear covariant functor and the $i$-th local cohomology functor $H^i_I$ is defined to be its $i$-th right derived functor. We call $H^i_I(M)$ the $i$-th local cohomology module of $M$ supported at the ideal $I$. If $I$ is generated by $f_1,\dots, f_t$, then $H^i_I(M)$ is the $i$-th cohomology of the \v Cech complex
 $$ 0 \rightarrow M \rightarrow \displaystyle {\bigoplus^t_{i=1}} M_{f_i} \rightarrow \displaystyle {\bigoplus_{1\leq i<j \leq t}} M_{f_if_j} \rightarrow \cdots \rightarrow M_{f_1\cdots f_t} \rightarrow 0. $$ 
 
\subsection{Toric Algebras}
 We introduce some notations for later use. For more information on toric algebras, the reader is referred to \cite{Fu}, \cite{MS} or \cite{24h}. 
 
 Let $A$ be a $d\times n$ integer matrix with columns $a_1,\dots,a_n$. Assume $\ZZ A= \ZZ^d$. 
 For a field $\KK$, the semigroup subring $S_{A,\KK} := \KK[\NN A] = \KK[t^{a_1},\dots,t^{a_n}]$ of the Laurent polynomial ring $\KK[\ZZ^d] =\KK[t_1^{\pm 1},\dots, t_d^{\pm 1}]$ 
 is called the toric algebra associated to the matrix $A$. Denote $\widetilde{\NN A} := \RR_{\geq 0} A \cap \ZZ^d $ as the saturation of $\NN A$. 
 Then $\widetilde{S_{A,\KK}}:= \KK[\widetilde{\NN A}]$ is the normalization of $S_{A,\KK}$.
 
\subsection{Rings of Differential Operators} \label{diff}
 For a commutative algebra $R$ over a field $\KK$, set 
 $D_0(R;\KK) := R$ and for $ i >0$ $$D_i(R;\KK) := \{ f \in \Hom_{\KK}(R,R) \mid [f,r] \in D_{i-1}(R;\KK) \text{ for all } r \in R \}.$$  
 Then the ring of differential operators is defined to be $$ D(R;\KK) := \displaystyle{\bigcup_j D_j(R;\KK)}.$$ 
 When $R$ is a polynomial ring over a field $\KK$ of characteristic $0$, $D(R;\KK)$ is the usual Weyl algebra. 
 In this paper, a module over $D(R;\KK)$ means a left $D(R;\KK)$-module.
 
 \begin{lemma} If $M$ is a $D(R;\KK)$-module and $f \in R$ then the $R$-module structure on $M_f$ extends uniquely to a $D(R;\KK)$-module 
 structure such that the natural map $M \rightarrow M_f$ is a $D(R;\KK)$-module homomorphism. 
 In particular, via the \v Cech complex $H^i_I(M)$ has a natural $D(R;\KK)$-module structure. 
 \end{lemma}
 \begin{proof} See \cite{Lyu00}, Example (b).
 \end{proof}
 
 When $R$ is a regular algebra over a field $\KK$ of characteristic $0$, $D(R;\KK)$ is well understood (see e.g. \cite{Bj}). In this case, 
 the local cohomology modules $H^i_I(R)$ are holonomic as $D(R;\KK)$-modules and hence are of finite length (see \cite{Lyubeznik}). 
 This essential property enables Lyubeznik to achieve many finiteness results of the local cohomology modules. 

Unfortunately, $D(R;\KK)$ does not behave well when $R$ is singular; we don't have a notion of holonomicity in this case. 
This complicates the study of $H^i_I(R)$ via the theory of $D$-modules. 
On the bright side, when $R= S_{A,\KK}$ is a toric algebra over an algebraically closed field $\KK$ of characteristic $0$, 
there is a nice combinatorial structure for $D(R;\KK)$ which we will present in the next subsection. 
Our finiteness results about local cohomology modules substantially rely on this structure.

\subsection{Rings of Differential Operators over Toric Algebras}
In the rest of this paper, we denote $S_A := S_{A,\KK}$ where $\KK$ is an algebraically closed field of characteristic $0$.
Following \cite{Saito-Traves01}, the noncommutative ring $D_A:= D(S_A,\KK)$ can be described as a $\ZZ^d$-graded subring of 
$$D(\KK[\ZZ^d];\KK)= \KK[t_1^{\pm 1},\dots, t_d^{\pm 1}]\langle \del_1,\dots, \del_d \rangle$$ where $[\del_i, t_j]=\delta_{ij}$, $[\del_i, t_j^{-1}]=-\delta_{ij}t_j^{-2}$, 
and the other pairs of variables commute. More precisely, with the notation $\theta_i := t_i\del_i$, one has
 $$D_A = \displaystyle{\bigoplus_{a\in \ZZ^d} t^a \II(\Omega(a))}$$ where $\Omega(a) = \NN A \setminus (-a+ \NN A)$ and $\II(\Omega(a))$ is the vanishing ideal of $\Omega(a)$ 
 in $\KK[\theta_1, \dots, \theta_d]$.

\section{Finiteness properties of $H^i_I(S_A)$} \label{finiteness}

 In this section, $\calF$ will be denoted to be the set of all facets of $\RR_{\geq 0}A$. 
 For a face $\tau$ of $\RR_{\geq 0}A$, we write $\NN (A \cap \tau):= \NN A \cap \RR \tau$ and denote $\ZZ(A \cap \tau)$ as the group generated by $\NN( A \cap \tau)$. 
 
 We recall some notations in \cite{Saito-Traves01}, which are crucial to the proofs of Theorems \ref{cyclic} and \ref{finite length}.\\
 For $a \in \ZZ^d$ and $\tau$ a face of $\RR_{\geq 0}A$, define
               $$E_{\tau}(a) := \{ l \in \CC (A \cap \tau) \mid a-l \in \NN A + \ZZ(A \cap \tau) \}/ \ZZ (A \cap \tau).$$
               Notice that 
               $$E_{\tau}(a) \subseteq [ \ZZ^d \cap \CC(A \cap \tau) ]/\ZZ (A \cap \tau) = \widetilde{\ZZ (A \cap \tau)}/ \ZZ(A \cap \tau).$$ 
               Here $\widetilde{\ZZ (A \cap \tau)} := \CC(A\cap \tau) \cap \ZZ^d$ is the saturation of $\ZZ(A\cap \tau)$, so 
               each $E_{\tau}(a)$ is a finite set.
               Consider  the ordering $\leq$ on $\ZZ^d$ defined by 
	       \[[a \leq b] \Longleftrightarrow [E_{\tau}(a) \subseteq E_{\tau}(b)] \text{ for all faces } \tau \text{ of } \RR_{\geq 0}A.\] 
	       This ordering induces an equivalence relation on 
               $\ZZ^d$ by 
	       \[[a \sim b] \Longleftrightarrow [a \leq b \text{ and } b\leq a].\] 
 Now, we are ready for the first main theorem.
 \begin{theorem} \label{cyclic} $S_A[f^{-1}] = D_A \cdot f^{-1}$ is cyclic as a left $D_A$-module for any monomial $f \in S_A$.
 \end{theorem}
 \begin{proof} It's clear that $S_A[f^{-1}] \supseteq D_A \cdot f^{-1}$. Conversely, write $f=t^b$ for some $b \in \NN A$ (we may assume $b \neq 0$). 
 Since $S_A [f^{-1}] \subseteq D_A \cdot \{ t^{-mb} \mid m \in \NN \}$, it suffices to show that $t^{-mb} \in D_A \cdot t^{-b}$ for all $m \in \NN$.
 By Proposition 4.1.5(1) in \cite{Saito-Traves01}, it is enough to prove that $-b \sim -mb$. Indeed, we show that, for any face $\sigma$, $E_{\sigma}(-b)=E_{\sigma}(-mb)$ 
 as follows.\\
 (1) Let $\sigma$ be a facet, and suppose $b \notin \sigma$. Then $F_{\sigma}(-b) < 0 $, and hence $F_{\sigma}(-mb) < 0$. So $E_{\sigma}(-b)=E_{\sigma}(-mb) = \emptyset$.\\
 (2) Let $\tau$ be a face and suppose that $b \notin \tau$. Then there exists a facet $\sigma$ containing $\tau$ with $b \notin \sigma$.
     Hence $E_{\sigma}(-b)=E_{\sigma}(-mb) = \emptyset$ by (1), and thus $E_{\tau}(-b)=E_{\tau}(-mb) = \emptyset$.\\
 (3) Let $\tau$ be a face and suppose $b \in \tau$. Then $b \in \NN A \cap \tau = \NN(A \cap \tau)$. Hence, $\pm b \in \ZZ(A \cap \tau)$, 
     and thus $E_{\tau}(-b)=E_{\tau}(-mb)$ by definition.
 \end{proof}

 \begin{remark} \label{noeth} Via the \v Cech complex, it follows immediately from Theorem~\ref{cyclic} that $H^i_I(S_A)$ is finitely generated over $D_A$ if $D_A$ is left Noetherian.
 The left Noetherianess of $D_A$ was studied by Saito and Takahashi \cite{Saito-Takahashi09}. They proved that $D_A$ is 
 left Noetherian if $S_A$ satisfies Serre's condition ($S_2$). Serre's condition is, by \cite{Ishida}, equivalent to $$S_A = \bigcap_{\tau  \text {: facets}} \KK[\NN A +\ZZ(A \cap \tau)]. $$
 Saito and Takahashi also gave a necessary condition (on $S_A$) for $D_A$ to be left Noetherian. However, $D_A$ is not always left Noetherian.
 \end{remark}

 Nonetheless, we have
 \begin{theorem}\label{finite length} For any $i$ and any monomial ideal $I$, $H^i_I(S_A)$ is of finite length as a $D_A$-module.
 \end{theorem}
 \begin{proof} In view of the \v Cech complex, since any localization $S_A[f^{-1}]$ (with monomial $f$) is a $D_A$-submodule of $\KK[\ZZ^d]$, 
               it suffices to show that $\KK[\ZZ^d]$ has a composition series.\\
               Consider the notations in the beginning of this section.
               For $a,b \in \ZZ^d$, we will write $a < b $ if $a \leq b$ but $a \nsim b$.            
               Then, for each $a \in \ZZ^d$, $ \bigoplus_{b\geq a} \KK t^b$ is generated by $t^a$ as a $D_A$-module.
               Moreover, $$ \frac{\bigoplus_{b\geq a}\KK t^b} {\bigoplus_{b > a} \KK t^b} \cong \bigoplus_{b \in [a]} \KK t^b$$
               is, by Theorem~4.1.6 in \cite{Saito-Traves01}, a simple $D_A$-module where $[a] = \{b\in \ZZ^d \mid b \sim a \}.$
               Since there are only finitely many faces and since each $E_{\tau}(a)$ is contained in $\widetilde{\ZZ (A \cap \tau)}/ \ZZ(A \cap \tau)$, 
               there are finitely many equivalence classes determined by $\sim$; we denote them $[\alpha_1],[\alpha_2],\dots,[\alpha_k]$.  
               We may rearrange the order so that, for any pair $i < j$, either $\alpha_i > \alpha_j$ or $\alpha_i$ and $\alpha_j$ are incomparable. 
               Denote $T_a :=\bigoplus_{b\geq a} \KK t^b$. Then the filtration
               $$ 0 \subsetneq T_{\alpha_1} \subsetneq \dots \subsetneq \Sigma_{l=1}^i T_{\alpha_l} \subsetneq \dots \subsetneq \Sigma_{l=1}^k T_{\alpha_l} = \KK[\ZZ^d]$$
               is a composition series of $D_A$-submodules of $\KK[\ZZ^d]$.
 \end{proof}
 
 \begin{example} 
 For $1$-dimensional $S_A$, the composition series of $\KK[\ZZ]$ is easy to describe.
 In this case, $\RR_{\geq 0} A$ has two faces, $0$ and $\sigma = \RR_{\geq 0} A$.
 For $a \in \ZZ$, 
 \[ E_0(a) = \{ \ell \in \{0\} \mid a - \ell \in \NN A\} / \{0 \} , \text{ and}\]
 \[ E_{\sigma}(a) = \{ \ell \in \ZZ \mid a - \ell \in \ZZ\}/ \ZZ.\]
 Thus $E_0(a) = \{0\} \text{ if } a \in \NN A$, $E_0(a) = \emptyset \text{ if } a \notin \NN A$, and $E_{\sigma}(a) = \{0\} \text{ for all }a \in \ZZ.$
 Therefore, $[0]$ and $[-1]$ are the two equivalence classes determined by $\sim$ and we have the composition series
 \[ 0 \subsetneq \KK[\NN A] = T_0 \subsetneq T_{-1}= \KK[\ZZ].\]
 \end{example} 
 
 \begin{remark} Theorem~\ref{cyclic} and Theorem~\ref{finite length} also hold for any field $\KK$ with characteristic $0$ by the isomorphism 
 $$ H^i_I(S_{A,\KK}) \otimes \overline{\KK} \cong H^i_I(S_{A,\overline{\KK}}).$$
 \end{remark}
  
 \begin{remark} \label{max}
 Suppose $I= \frakm$, the maximal graded ideal of $S_A$. Here we assume that the semigroup $\NN A$ is pointed, so that $0$ is the only invertible element.
 \begin{enumerate}
 \item Recall that $H^i_{\frakm}(S_A)$ can be computed as the $i$-th cohomology of the Ishida complex \cite{Ishida} or \cite{24h}.
       Therefore, $H^1_{\frakm}(S_A)$ is finitely generated as an $S_A$-module. 
       Indeed, it suffices to observe that
               $$ H^1_{\frakm}(S_A)  \twoheadleftarrow \bigcap_{\sigma \text {: rays}} \KK[\NN A + \ZZ(A\cap \sigma)]
               \subseteq \bigcap_{\tau \text {: facets}} \KK[\NN A + \ZZ(A\cap \tau)] \subseteq \widetilde{S_A} $$
               and the fact that $\widetilde{S_A}$ is finite over $S_A$. \\            
               Moreover, $H^d_{\frakm}(S_A)$ is cyclic as a left $D_A$-module.
               This is because the $d$-th module in the Ishida complex is $\KK[\ZZ^d]$ 
               which is cyclic by Theorem~\ref{cyclic}.

 \item In general, Sch\"afer and Schenzel \cite{Schafer-Schenzel90} showed that there is a partition of $\ZZ^d$ with respect to 
 which $H^i_{\frakm}(S_A)$ can be written as a finite direct sum of $\KK$-vector spaces. This decomposition coincides with the sector partition 
 appearing in \cite{Matusevich-Miller} (see also \cite{HM05} for a more general notion of sector partition).
 More precisely, let $\conv(A)$ be the set of all faces of $\RR_{\geq 0} A$ and for any filter (cocomplex) $\nabla$ of $\conv(A)$, denote 
 $$P_{\nabla} = \bigcap_{A \cap \tau \in \nabla} [\NN A + \ZZ(A \cap \tau)] \setminus
 \bigcup_{A \cap \tau \notin \nabla} [\NN A + \ZZ(A \cap \tau)].$$
 Then the $P_{\nabla}$'s form a partition (sector partition) of $\ZZ^d$ and
 $$ H^i_{\frakm}(S_A) = \bigoplus_{\nabla} \KK[P_{\nabla}] \otimes_{\KK} H^i(\conv(A),\conv(A) \setminus \nabla; \KK).$$
 On the other hand, for $a \in \ZZ^d$ denote 
   $$\nabla(a) := \{ \text{ face }\tau \text{ of }\RR_{\geq0}A \mid a \in \NN A +\ZZ(A \cap \tau) \}$$ 
   and consider the equivalence
 relation $a \equiv a' \Leftrightarrow \nabla(a)=\nabla(a')$. Then $P_{\nabla(a)}$ is the equivalence class containing $a$. Notice that $P_{\nabla}$ could be empty and 
 that $[a \in P_{\nabla}] \Leftrightarrow [P_{\nabla}= P_{\nabla(a)}]$.
   
 Theorem 6 in \cite{Matusevich-Miller} shows that the partition determined by the equivalence relation $\sim$ in the proof of Theorem~\ref{finite length} is finer than the sector partition
 determined by $\equiv$. 
 \end{enumerate}
 \end{remark}
 
 Notice that each $\KK[P_{\nabla}]$ is naturally a left $\ZZ^d$-graded $D_A$-module, because each $\KK [ \NN A + \ZZ(A\cap \tau)]$ is.
 If $S_A$ is normal, the $D_A$-module $\KK[P_{\nabla}]$ is simple. In fact, we have
 
 \begin{theorem} \label{sector-normal} If $S_A$ is normal and $I$ is a monomial ideal in $S_A$, then every simple subquotient of $H^i_I(S_A)$ is of the form $\KK[P_{\nabla}]$ coming from the sector partition.
 \end{theorem}
 \begin{proof} By Theorem~\ref{finite length} and Remark~\ref{max}(2), we only have to show that $\sim$ and $\equiv$ define the same equivalence relation on $\ZZ^d$.
 Note that the normality of $S_A$ implies that 
  \[ \begin{aligned}
      E_{\tau}(a) &= \{0\} \text{, if } a \in \NN A + \ZZ(A \cap \tau), \text{ and}\\
      E_{\tau}(a) &= \emptyset \text{ if } a \notin \NN A + \ZZ(A \cap \tau).
      \end{aligned}
      \] 
 So we have
 $$\begin{aligned}
                      & a \sim b  \\
    \Leftrightarrow   & E_{\tau}(a) = E_{\tau}(b) \text{ for all } \tau \\
    \Leftrightarrow   & E_{\tau}(a) = \{0 \} \text{ if and only if } E_{\tau}(b) = \{0 \} \\
    \Leftrightarrow   & a \in \NN A + \ZZ(A \cap \tau) \text{ if and only if } b \in \NN A + \ZZ(A \cap \tau) \\
    \Leftrightarrow   & a \equiv b. 
 \end{aligned}$$
 Therefore, the simple subquotients of $\KK[\ZZ^d]$ are precisely the $D_A$-modules
 $$\KK[P_{\nabla}] = \bigoplus_{b \in [a]} \KK t^b \cong \frac{\bigoplus_{b\geq a} \KK t^b}{\bigoplus_{b > a} \KK t^b}.$$
 \end{proof}

 \begin{example}\label{2dim} Consider $A= \bigl( \begin{smallmatrix} 1&1&1 \\ 0&1&2 \end{smallmatrix} \bigr)$. Then $S_A = \KK[t,ts,ts^2]$ is a $2$-dimensional normal toric algebra.
 Let $I=(ts)$ be the ideal of $S_A$ generated by $ts$. We shall describe a composition series of $H^1_I(S_A)$. By \v Cech complex, $H^1_I(S_A) = \KK[\ZZ^d] / S_A$.
 Following the notations in the proof of Theorem~\ref{finite length} and Remark~\ref{max},
  let $$ \begin{aligned} &\nabla_0 = \{ 0, \sigma_1, \sigma_2, \RR_{\geq0}A \} \text{, } \nabla_1 = \{ \sigma_1, \RR_{\geq 0} A \}  
 \text{, }  \nabla_2 = \{ \sigma_2, \RR_{\geq 0} A \},  \\
       &\nabla_{12} = \{ \sigma_1, \sigma_2,  \RR_{\geq 0} A \} \text{, } \nabla_{A} = \{ \RR_{\geq 0} A \} \end{aligned},$$ where 
       $\sigma_1 = \RR_{\geq 0} \bigl(  \begin{smallmatrix} 1 \\ 0 \end{smallmatrix} \bigr)$
       and $\sigma_2 = \RR_{\geq 0} \bigl(  \begin{smallmatrix} 1 \\ 2 \end{smallmatrix} \bigr).$
 Then $$ \begin{aligned} 
        &P_{\nabla_0} = P_{\nabla(a_0)} = \NN A ,\\
	&P_{\nabla_1} = P_{\nabla(-a_1)} = [\NN A + \ZZ(A \cap \sigma_1)] \setminus [\NN A + \ZZ(A \cap \sigma_2)], \\
         &P_{\nabla_2} = P_{\nabla(-a_2)}= [\NN A + \ZZ(A \cap \sigma_2)] \setminus [\NN A + \ZZ(A \cap \sigma_1)], \\
         &P_{\nabla_{12}} = \emptyset \text{, and }\\
	 &P_{\nabla_A} =P_{\nabla(-a_3)}= \ZZ^2 \setminus  \bigl[(\NN A + \ZZ(A \cap \sigma_1)) \cup (\NN A + \ZZ(A \cap \sigma_2))\bigr]
	 \end{aligned}
	 $$ 
	 where $a_0 =  \bigl(  \begin{smallmatrix} 0 \\ 0 \end{smallmatrix} \bigr)$ and $a_i$, $i=1,2,3$, is the $i$th column of $A$.
 In terms of notation in Theorem~\ref{finite length},
      $$ \begin{aligned}
         T_{a_0} &= \KK[P_{\nabla_0}]= S_A \\
	 T_{-a_1} &= \KK[P_{\nabla_0}] \oplus \KK[ P_{\nabla_1}]\\
	 T_{-a_2} &= \KK[P_{\nabla_0}] \oplus \KK[ P_{\nabla_2}]\\
	 T_{-a_3} &= \KK[\ZZ^2] = \KK[P_{\nabla_0}] \oplus \KK[P_{\nabla_1}] \oplus \KK[ P_{\nabla_2}] \oplus \KK[P_{\nabla_A}]
	 \end{aligned}
	 $$
  So $ 0 \subset T_{a_0} \subset T_{-a_1} \subset T_{-a_1}+T_{-a_2} \subset T_{-a_3}= \KK[\ZZ^2]$ is a composition series of $\KK[\ZZ^2]$.
  Quotienting out $S_A$, we obtain a composition series of $H_I^1(S_A)$:
   $$ 0 \subset \KK[ P_{\nabla_1}]  \subset \KK[P_{\nabla_1}] \oplus \KK[ P_{\nabla_2}] 
     \subset \KK[P_{\nabla_1}] \oplus \KK[ P_{\nabla_2}] \oplus \KK[P_{\nabla_A}] .$$
 \end{example}
 
 \begin{example} \label{hartshorne} This example is essentially due to Hartshorne \cite{Ha}. 
 We adopt its combinatorial description which can be found in \cite{24h} or \cite{MS}.
 
 Consider $A= \Bigl( \begin{smallmatrix} 1&1&1&1 \\ 0&1&0&1 \\ 0&0&1&1 \end{smallmatrix} \Bigr)$. Then $S_A = \KK[r,rs,rt,rst]$ is a normal toric algebra. 
 Consider the ideal $I=(r,rs)$ of $S_A$. Then the socle Hom$_{S_A}(S_A/ \frakm , H^2_I(S_A))$ is infinite dimensional where $\frakm=(r,rs,rt,rst)$ is the maximal graded 
 ideal of $S_A$. 
 However, according to Theorem~\ref{finite length} $H^2_I(S_A)$ is of finite length over $D_A$. 
 
 In fact, $H^2_I(S_A)$ is $D_A$-simple. To see this, using the notation in Remark~\ref{max} we consider the filter $\nabla = \{ \sigma_{12} , \RR_{\geq0} A \}$ where 
 $\sigma_{12} =\RR_{\geq 0 } \Bigl( \begin{smallmatrix} 1&1\\ 0&1 \\ 0&0 \end{smallmatrix} \Bigr) $ is 
 a facet of $\RR_{\geq 0} A$. 
 We claim that $$P_{\nabla} = [\NN A + \ZZ(A\cap \sigma_{12})] \setminus [(\NN A + \ZZ a_1)\cup(\NN A + \ZZ a_2)]$$
 where $a_1= \Bigl( \begin{smallmatrix} 1\\ 0 \\ 0 \end{smallmatrix} \Bigr)$ and $a_2= \Bigl( \begin{smallmatrix} 1\\ 1 \\ 0 \end{smallmatrix} \Bigr)$.
 Therefore, in view of the \v Cech complex we have the isomorphism $$H^2_I(S_A) \cong \KK[P_{\nabla}]$$ which is $D_A$-simple by Theorem~\ref{sector-normal}.
 
 The claim is equivalent to the equality
 $$ \begin{aligned}  &\bigl[\NN A + \ZZ(A\cap \sigma_{12}) \bigr] \cap \Bigl[\bigcup_{\sigma = \sigma_{13},\sigma_{24},\sigma_{34}}(\NN A + \ZZ(A\cap \sigma))\Bigr] \\
     = & \bigl[\NN A + \ZZ(A\cap \sigma_{12}) \bigr] \cap \bigl[(\NN A + \ZZ a_1)\cup(\NN A + \ZZ a_2) \bigr]
     \end{aligned}$$
 where $\sigma_{13} =\RR_{\geq 0 } \Bigl( \begin{smallmatrix} 1&1\\ 0&0 \\ 0&1 \end{smallmatrix} \Bigr) $,
       $\sigma_{24} =\RR_{\geq 0 } \Bigl( \begin{smallmatrix} 1&1\\ 1&1 \\ 0&1 \end{smallmatrix} \Bigr) $, and
       $\sigma_{34} =\RR_{\geq 0 } \Bigl( \begin{smallmatrix} 1&1\\ 0&1 \\ 1&1 \end{smallmatrix} \Bigr) $.
 This equality can be verified by the following data:
 $$\begin{aligned}
   &\NN A+\ZZ a_1 = \{^t(x,y,z) \in \ZZ^3 \mid y\geq 0 \text{ and} z\geq 0 \}, \\ 
   &\NN A+\ZZ a_2 = \{^t(x,y,z) \in \ZZ^3 \mid x \geq y \text{ and} z\geq 0 \}, \\
   &\NN A+\ZZ (A \cap \sigma_{12})  = \{^t(x,y,z) \in \ZZ^3 \mid z \geq 0 \}, \\
   &\NN A+\ZZ (A \cap \sigma_{13})  = \{^t(x,y,z) \in \ZZ^3 \mid y \geq 0 \},\\
   &\NN A+\ZZ (A \cap \sigma_{24})  = \{^t(x,y,z) \in \ZZ^3 \mid x \geq y \}, \\
   &\NN A+\ZZ (A \cap \sigma_{34})  = \{^t(x,y,z) \in \ZZ^3 \mid x \geq z \}.
   \end{aligned}
   $$ 
  \end{example}
 
 \begin{remark} Helm and Miller \cite{HM03} studied the Bass numbers of local cohomology modules over toric algebras. 
 As a generalization of Hartshorne's example, their main result (\cite{HM03}, Theorem~7.1) implies that for a Gorenstein normal toric algebra $S_A$, 
 $\NN A$ is not simplicial if and only if 
 there exists a $\NN A$-graded prime $\frakp$ of dimension $2$ such that $H^{d-1}_{\frakp}(S_A)$ has infinite dimensional scole.
 \end{remark}
 
\section{Associated graded rings $\gr D_A$ and characteristic cycles} \label{charcycle}

 \subsection{Associated graded rings $\gr D_A$} \label{subsec-gr}
 
 Let $R$ be a $\KK$-algebra as in subsection~\ref{diff}. 
 The definition of $D(R;\KK)$ gives a order filtration of $D(R;\KK)$:
 $$ 0 \subseteq R =D_0(R;\KK) \subseteq D_1(R;\KK) \subseteq D_2(R;\KK) \subseteq \dots .$$
 Define the associated graded ring of $D(R;\KK)$ to be:
 $$ \gr D(R;\KK) := D_0 \oplus (D_1/D_0) \oplus (D_2/D_1) \oplus \dots $$
 where $D_i:= D_i(R;\KK)$. 
 From the definition of $D(R;\KK)$, $\gr D(R;\KK)$ is a commutative $R$-algebra and we have the natural embedding $R \hookrightarrow \gr D(R;\KK)$.
 For example, if $R = \KK[t_1,\dots,t_d]$ is a polynomial algebra over $\KK$, then $D(R;\KK) = \KK[t_1,\dots,t_d] \langle \del_1,\dots, \del_d \rangle$ is the Weyl algebra.
 The associated graded ring $\gr D(R;\KK)= \KK [ t_1,\dots,t_d, \xi_1,\dots, \xi_d ]$ is a $2d$-dimensional polynomial algebra over $\KK$ where
 $\xi_i$ is the image of $\del_i$ in the associated graded ring $\gr D(R;\KK)$.
 In what follows, we will use the description: $$\gr D(\KK[t_1^{\pm 1},\dots,t_d^{\pm 1}];\KK) = \KK[ t_1^{\pm 1},\dots,t_d^{\pm 1}, \Theta_1,\dots, \Theta_d ]$$ where 
 $\Theta_i = t_i \xi_i$ is the image of
 $\theta_i=t_i \del_i$ in the associated graded ring $\gr D(R;\KK)$.

 When $R$ is a regular algebra over $\KK$, $\Spec(\gr D(R;\KK))$ can be identified as the cotangent bundle of the variety $\Spec R$ with the projection 
 $$\pi : \Spec(\gr D(R;\KK)) \rightarrow \Spec R$$ induced by the embedding $R \hookrightarrow \gr D(R;\KK)$. The fiber of $\pi$ over a closed point of $\Spec R$ is 
 the cotangent space over that point which is isomorphic to the affine space $\KK^{\dim R}$.
 
 In this subsection, we discuss the the map $\pi$ for a certain class of toric algebras. We shall see that in some cases the fibers of $\pi$ behave nicely
 (Theorems~\ref{fiber}, \ref{fiber2}). We also give an example (Example~\ref{badfiber}) of more complicated nature.
 
 To begin with, consider the natural order filtration of $D_A$ inherited from that of $D(\KK[\ZZ^d])$. With respect to this filtration, one can regard gr$D_A$ as a commutative
 subalgebra of $\text{gr} D(\KK[\ZZ^d]) = \KK[t_1^{\pm 1},\dots, t_d^{\pm 1}, \Theta_1,\dots, \Theta_d ].$
 When $\gr D_A$ is finitely generated over $\KK$, Musson showed that it has dimension $2d$ \cite{Mu87}. 
 Saito and Traves proved that $\gr D_A$ is finitely generated over $\KK$ if and only if the semigroup $\NN A$ is scored \cite{Saito-Traves04}. 
 By definition, a semigroup $\NN A$ is scored if $$ \NN A = \bigcap_{\sigma \text{: facet}} \{a \in \ZZ^d \mid F_{\sigma}(a) \in F_{\sigma}(\NN A) \},$$
 or equivalently, $\widetilde{\NN A} \setminus \NN A$ is a union of finitely many hyperplane sections parallel to some facets of $\RR_{\geq 0} A$.
 The scored condition implies Serre's condition ($S_2$) (see Remark~\ref{noeth}).
 
 We should remark that if $S_A$ is normal, then $\gr D_A$ is Gorenstein (\cite{Mu87} Theorem D). In general, even for $1$-dimensional semigroup ring (which is always scored), the associated graded ring can have bad singularities. In fact, we have the following
 \begin{prop} \label{notCM}
 If $S_A$ is a $1$-dimensional toric algebra which is not normal, then $\gr D_A$ is not Cohen-Macaulay.
  \end{prop}
  \begin{proof}
  Using the formula in Lemma~\ref{gr}, we see that $\gr D_A$ is again a toric algebra over $\KK$. 
  Indeed, notice that $0$ is the only facet and that $n_{0,w}= \lvert \Omega(w) \rvert $. Note also that $\lvert \Omega(-w) \rvert = w +\lvert \Omega(w) \rvert$ by
  Lemma~\ref{fiber of gr}. So by Lemma~\ref{gr},  
  \[ \gr D_A = \KK \left[ t \xi, t^{\mid \Omega(w) \mid} \xi^{\mid \Omega(-w) \mid} : \lvert w \rvert \in \{a_1,\dots,a_n \} \cup \Hole (\NN A) \right]. \]
  Therefore $\gr D_A$ is a two dimensional toric algebra over $\KK$.
  
  We claim that \begin{equation}\label{finitecodim} \text{dim}_{\KK} \frac{\KK[t, \xi]}{\gr D_A} < \infty. \end{equation}
  Take $\ell$ to be the maximal number of $2 \lvert \Omega(-w) \rvert$ for $w \in \Hole (\NN A)$.
  To prove \eqref{finitecodim}, it is enough to show that $t^u \xi^v \in \gr D_A$ for all pairs $u,v \in \NN$ satisfying
  $u+v \geq \ell$. 
  Since $[t^u \xi^v \in \gr D_A \Leftrightarrow t^v \xi^u \in \gr D_A]$, by symmetry we may assume $w_0:=u-v \geq 0$.
  \begin{itemize}
  \item If $w_0 \in \Hole(\NN A)$, 
   $ 2u \geq u+v \geq \ell \geq 2 \lvert \Omega(-w_0) \rvert.$
   So $(t \xi)^{u-\lvert \Omega(-w_0) \rvert} \in \gr D_A$, and hence
  $ t^u \xi^v = (t \xi)^{u-\lvert \Omega(-w_0) \rvert} t^{\lvert \Omega(-w_0) \rvert} \xi^{\lvert \Omega(w_0) \rvert} \in \gr D_A.$
  
  \item If $w_0 \in \NN A$ we have $t^u \xi^v =(t \xi )^v t^{w_0} = (t \xi )^v t^{\lvert \Omega(-w_0) \rvert}  \in \gr D_A.$ 
  \end{itemize}
  So the claim is proved. Now, applying the criterion in Remark~\ref{noeth} to the toric algebra $\gr D_A$,
  we see that $\gr D_A$ doesn't satisfy  Serre's condition ($S_2$). Hence, $\gr D_A$ is not Cohen-Macaulay.
  \end{proof}
  
  \begin{remark} The claim \eqref{finitecodim} holds true for any affine curve with injective normalization (see the proof of Theorem 3.12 in \cite{SS88}). 
  In fact, this codimension is known to be the Letzter--Makar-Limanov invariant, which plays an important role in the theory of Calogero-Moser space. For more information, see the work by Berest and Wilson \cite{BW}.
  \end{remark}
  
  Now, let $\frakm$ be the maximal graded ideal of $S_A$ corresponding to the closed point $0$ of the toric variety $\Spec(S_A)$.
  Let $I = \sqrt{\frakm \text{gr}D_A}$ be the radical of the extended ideal of $\frak m$ under the embedding $S_A \hookrightarrow \text{gr}D_A$.
  We are going to show that if $\NN A$ is simplicial and scored then $\text{gr}D_A / I $ is isomorphic to $S_A$ as $\KK$-algebras. This implies that the reduced induced structure
  of the fiber of $\pi : \Spec(\text{gr}D_A) \rightarrow \Spec(S_A)$ over
  the point $0$ is isomorphic to the ambient toric variety. 
 
  The following two lemmas are needed:
  
  \begin{lemma}\label{gr}\cite{Saito-Traves04} 
             For scored $\NN A$,
	     $$\gr D_A = \bigoplus_{a \in \ZZ^d } t^a \KK[\Theta_1,\dots,\Theta_d]\cdot P_a \text{ where} $$
             $$P_a = \displaystyle{\prod_{\sigma \in \calF}} F_{\sigma}(\Theta)^{n_{\sigma,a}} \text{ and } n_{\sigma,a} = \# \{ F_{\sigma}(
             \NN A) \setminus [-F_{\sigma}(a) + F_{\sigma}( \NN A)] \}. $$
 \end{lemma}
 
 \begin{lemma}\label{fiber of gr} Let $\NN A$ be scored. Then for any $a \in \ZZ^d$ and $\sigma \in \calF$, $$n_{\sigma,-a} = n_{\sigma,a} + F_{\sigma}(a).$$ 
 In particular, 
 \begin{enumerate}
 \item if $\sigma$ is a facet with the property that $F_{\sigma}(a) \leq 0$, then $n_{\sigma, ka} \leq k\cdot n_{\sigma,a}$ for large $k \in \NN$
       and furthermore $P_a^k = P_{ka} \cdot P$ for some $P \in \KK[\Theta]$;
 \item if $F_{\sigma}(a) \leq 0$ for all $\sigma \in \calF$ and $-a \notin \NN A$, then $n_{\sigma, ka} < k\cdot n_{\sigma,a}$ for some $\sigma \in \calF$ and large $k \in \NN$.
 \item if $a \in \NN A$, then $n_{\sigma,-a}=F_{\sigma}(a)$.
 \item if $F_{\sigma}(a) > 0$, then $n_{\sigma, ka} = 0 $ for large $k \in \NN$.
 \end{enumerate}
 \end{lemma}
 
 \begin{proof} To prove $n_{\sigma,-a} = n_{\sigma,a} + F_{\sigma}(a)$ for any $a \in \ZZ^d$ and $\sigma \in \calF$, it's enough to show the case where $F_{\sigma}(a) > 0$. 
 Set $N=F_{\sigma}(\NN A)$ and $n = F_{\sigma}(a) > 0$.
 Then
        $$\begin{aligned}
             n_{\sigma,-a} &= \# \{ N \setminus (n+ N ) \} \\
             &= \# \{b \in N \mid b-n \notin N \} \\
              & = n + \# \{ b \in N \mid b-n \notin N \text{ but } b-kn \in N \text{ for some } k \geq 2\} \\
              & = n + \# \{ c \in N \mid c+n \notin N\} \\
              &= n+ \#\{ N \setminus (-n + N) \} = n+ n_{\sigma,a}.
              \end{aligned}$$ Note that for the second equality we need the assumption that $\NN A $ is scored.
 
 Now, we prove the four additional statements:
 \begin{enumerate}
 \item Notice that $n_{\sigma, ka} = kF_{\sigma}(-a)$ for large $k$ and that $n_{\sigma,a} = n_{\sigma,-a} + F_{\sigma}(-a)$ where $n_{\sigma,-a} \geq 0$. 
  \item By assumption, $-a$ lies on a hyperplane parallel to some facet say $\sigma_0$. Then $n_{\sigma_0, -a} > 0$ and hence 
  $n_{\sigma_0, ka} < k\cdot n_{\sigma_0,a}$ by (1). 
 \item $a \in \NN A$ implies $n_{\sigma,a} = 0$. 
 \item This follows from the definition. Indeed, since $\NN A$ is scored, $(F_{\sigma}(ka) + \NN_0) \subseteq F_{\sigma}(\NN A)$ for large $k$. Then $F_{\sigma}(ka) + F_{\sigma}(\NN A) \subseteq F_{\sigma}(\NN A)$ and hence $n_{\sigma,ka} =0$.
 \end{enumerate}
 \end{proof}

 \begin{theorem}\label{fiber} If $\NN A$ is a simplicial scored semigroup, then $$ \gr D_A/I =\KK \left[\overline{t^{-a_i}\cdot P_{-a_i}} ; i=1,\dots,n \right] \cong S_A$$ 
 where $I=\sqrt{\frak m \gr D_A}$ and $\overline{t^{-a_i}\cdot P_{-a_i}}$ is the image of $t^{-a_i}\cdot P_{-a_i}$ in $\gr D_A/I$.
 \end{theorem}
 
 \begin{proof} We sketch how the proof of the left equality goes. Let $\calF = \{ \sigma_1,\dots, \sigma_d \}$ be the set of all facets of $\RR_{\geq 0} A$,
   and let 
   \[ C = -\widetilde{\NN A} = \{ a \in \ZZ^d \mid F_{\sigma}(a) \leq 0  \text{ for all } \sigma \in \calF\}. \]   
   We will prove the left equality in three steps. The first step shows $\Theta_i \in I$ for $i = 1,\dots,d$.
 The second step shows $t^a \cdot P_a \in I$ for all $a \in \ZZ^d \setminus C$.  
 Finally, the third step shows that $t^a \cdot P_a \in I$ if $a \in C \setminus (-\NN A)$, and $t^a \cdot P_a$ is a product of some $t^{-a_i} \cdot P_{-a_i}$, $i=1,\dots,n$, 
 if $a \in C \cap (-\NN A \setminus \{0\})$. 
 
 \begin{enumerate}
 \item For each $i=1,\dots,d$, consider the following subset of $\ZZ^d$: 
         $$ \{ F_{\sigma_i}(\Theta) = -1 \} \cap \left[ \bigcap_{j \neq i} \{ F_{\sigma_j}(\Theta) = 0 \} \right].$$
         Since $\NN A$ is simplicial, this is a one point set for each $i$, say $\{u_i \}$. 
	 Notice that since $t^{-u_i} \in I$, $ F_{\sigma_i}^{n_{\sigma_i,u_i}} = P_{u_i} = t^{-u_i}\cdot
	 t^{u_i} P_{u_i} \in I$ where $n_{\sigma_i,u_i} > 0$. Therefore, $F_{\sigma_i} \in I$ for each $i$. Since 
         $F_{\sigma_1},\dots,F_{\sigma_d}$ are linearly independent, we conclude that $\Theta_i \in I$ for $i = 1,\dots,d$.
 \item For $a \in \ZZ^d \setminus C$, $F_{\sigma}(a) >0$ for some $\sigma \in \calF$. By Lemma \ref{fiber of gr}(4), choose $k$ large so that 
       $$P_{ka} = \displaystyle{\prod_{F_{\sigma}(a) < 0 }} F_{\sigma}^{n_{\sigma, ka}}.$$ 
       Now, consider as in (1) the one-point set
       $$\left[  \bigcap_{F_{\sigma}(a) < 0} \{ F_{\sigma}( \Theta) = F_{\sigma}(ka) \} \right] \cap
       \left[ \bigcap_{F_{\sigma}(a) \geq 0} \{ F_{\sigma} (\Theta) =0\} \right] = \{b \} .$$
       We have $P_b = P_{ka}$ and $t^{ka-b} \in I$. By Lemma \ref{fiber of gr}(1) 
       $$ (t^aP_a)^k = t^{ka}P_a^k = t^{ka}P_{ka}\cdot P = (t^bP_b)\cdot P \cdot t^{ka-b} \in I.$$
       Therefore, $t^aP_a \in I$ as desired.
 \item Let $a \in C$.
 
       If $-a \notin \NN A$, by Lemma \ref{fiber of gr}(2) $(t^{a}P_a)^k = t^{ka}P_{ka} \cdot P $ for some nonconstant $P \in \KK[\Theta]$. Since $P$ is a product of some $F_{\sigma}$'s, $P \in I$ by (1) and hence $t^{a}P_a \in I$.
       
       If $-a \in \NN A \setminus \{0\}$, write $-a = \sum m_ia_i$. By Lemma \ref{fiber of gr}(3), 
       $$n_{\sigma,a} = F_{\sigma}(-a) = \displaystyle{\sum m_iF_{\sigma}(a_i)} = \displaystyle{\sum m_i n_{\sigma,-a_i}}$$ and therefore
       $$ t^aP_a =   t^{\sum m_i(-a_i)} \cdot \displaystyle{\prod_{\sigma \in \calF} F_{\sigma}^{\sum m_i n_{\sigma,-a_i}}} = \displaystyle{\prod
       (t^{-a_i}P_{-a_i} )^{m_i}}.$$
       
 \end{enumerate}

 To complete the proof, we establish the right isomorphism. First, recall that if $R\rightarrow S$ is a homomorphism of commutative rings 
 and $Q$ is a prime ideal in $S$ lying over a prime ideal $q$ of $R$, then dim$(S_Q/qS_Q) \geq \text{ht}Q-\text{ht}q$. 
 On the other hand, since $\gr D_A$ is finitely generated as a $\KK$-algebra which is also a domain, 
 each maximal ideal of $\gr D_A$ has height $2d$ (Here, we use the fact that $\dim \gr D_A = 2d$). Therefore, dim$(\text{gr}D_A/I) \geq d$. Now, consider the surjection from the polynomial ring $\KK[x_1,\dots,x_n]$ to $\KK \left[\overline{t^{-a_i}\cdot P_{-a_i}} ; i=1,\dots,n \right]$. By Lemma \ref{fiber of gr}(3), 
 $P_{-a_i} = \prod_{\sigma \in \calF} F_{\sigma}^{F_{\sigma}(a_i)}$. Observe that $\overline{t^{-a_i}\cdot P_{-a_i}}, i=1,\dots,n$, satisfy the relations in the toric ideal $I_A = \{ x^u-x^v | Au=Av \}$ (where for $u \in \ZZ^n$, $x^u:=x_1^{u_1}\cdots x_n^{u_n}$). Hence we have a surjection 
 $$S_A \cong \KK[x_1,\dots,x_n]/I_A \longrightarrow \KK \left[\overline{t^{-a_i}\cdot P_{-a_i}} ; i=1,\dots,n \right]$$ which is an isomorphism by comparing the dimensions.
 \end{proof}
 
 \begin{example}\label{badfiber} Consider $$A = \begin{pmatrix} 
                            1 & 0 & 0 & 1 \\
                            0 & 1 & 0 & 1 \\
                            0 & 0 & 1 & -1 
                            \end{pmatrix} .$$
             $S_A = \KK[s,t,u,stu^{-1}  ]$ is a $3$-dimensional normal toric algebra which is isomorphic to the toric algebra appearing in Example~\ref{hartshorne}. 
             By 4.1, 4.6, and 6.3 in \cite{Saito-Traves04}, 
             $$\begin{aligned}
                \gr D_A &= \KK[  s,t,u, stu^{-1}, \Theta_s, \Theta_t, \Theta_u, 
		 s^{-1} \Theta_s(\Theta_s+\Theta_u), t^{-1} \Theta_t(\Theta_t+\Theta_u), s^{-1}t^{-1}u
		 \Theta_s \Theta_t\\
                 & u^{-1}(\Theta_s+\Theta_u)(\Theta_t+\Theta_u), tu^{-1}(\Theta_s+\Theta_u),
		 t^{-1}u \Theta_t, su^{-1}(\Theta_t+\Theta_u), s^{-1}u \Theta_s ].
                 \end{aligned}$$
		 
	     Set $$\begin{aligned}
	        &a = s, b=t, c=u, d=stu^{-1},e= \Theta_s, f= \Theta_t, g= \Theta_u, h=s^{-1} \Theta_s(\Theta_s+\Theta_u), \\
		 &i=t^{-1} \Theta_t(\Theta_t+\Theta_u), 
		  j= u^{-1}(\Theta_s+\Theta_u)(\Theta_t+\Theta_u),
		  k=s^{-1}t^{-1}u\Theta_s\cdot \Theta_t,\\
                 &l=tu^{-1}(\Theta_s+\Theta_u),
		  m=t^{-1}u \Theta_t, 
		  n=su^{-1}(\Theta_t+\Theta_u), 
		  o=s^{-1}u \Theta_s
		 \end{aligned}
		 $$
		 Consider the surjection $\phi: \KK[a,\dots,o] \rightarrow \gr D_A$. 
		 Using Macaulay 2, we see that a primary decomposition of $\sqrt{\frak m \gr D_A}$ is the intersection of the two ideals
		 $(o, n, d, a, c, f + g, e, b, fj - il, fi + jm, f^2  + lm, fk + hm, hi - jk, fh - kl)$ and 
		 $(m, l, d, a, c, f, e + g, b, gk - io,gi + kn, g^2  + no, gj + hn, hi - jk, gh - jo)$ modulo Ker$ \phi$.
		 Therefore, the fiber $\pi^{-1}(0)$ has two components each of which is $4$-dimensional.
 \end{example}
 
 \begin{remark}\label{ref1} The left equality of Theorem~\ref{fiber} and Example~\ref{badfiber} can be achieved alternatively by a result of Saito. \\ 
 By Proposition 4.14 in \cite{S10}, 
 $ \pi^{-1}(\frakm) = \{ \mathfrak{P}(\mathfrak{q},\nu)\mid \nu \cap \RR_{\geq 0}A = \{0\} \}. $
 If $\NN A$ is simplicial, then $\pi^{-1} (\frakm) = \{ \mathfrak{P}(\frakm_0, -\RR_{\geq 0}A )\}$, which is the left equality of Theorem~\ref{fiber}.\\ 
 On the other hand, consider Example~\ref{badfiber}.
 If $\mathfrak{P}(\mathfrak{q}, \nu) \in \pi^{-1}(\frakm)$, then by \cite{S10} Proposition 4.14, 
 $ \mathfrak{q} \supseteq (\Theta_s, \Theta_t + \Theta_u) \text{ or } (\Theta_t, \Theta_s + \Theta_u). $
 If $\mathfrak{q} = (\Theta_s, \Theta_t + \Theta_u)$, then $\nu = \{\Theta_s \leq 0, \Theta_t + \Theta_u \leq 0\}$. If $\mathfrak{q} = (\Theta_t, \Theta_s + \Theta_u)$, then $\nu = \{\Theta_t \leq 0, \Theta_s + \Theta_u \leq 0\}$. They are the two minimal primes mentioned in Example~\ref{badfiber}.
 \end{remark}

 As a corollary of Theorem~\ref{fiber}, we can describe the fibers $\pi^{-1}(p)$ for every nonzero closed point $p$ in $\Spec S_A$.
 Let $p$ be in the torus orbit $O_{\tau}$ for some $e$-dimensional face $\tau$ of $\RR_{\geq 0} A$, so $p$ corresponds to a semigroup homomorphism
 $f_p: \NN A \rightarrow \KK$ with $f_p(a_i)=c_i$ where $c_i=0$ if and only if $a_i \notin \NN(A \cap \tau)$. Then $p$ corresponds to the maximal ideal
 $\frak m_p=(t^{a_1} - c_1,\dots, t^{a_n}-c_n)$ of $S_A$. The following theorem gives the reduced induced structure of $\pi^{-1}(p)$.
 
 \begin{theorem} \label{fiber2} Under the hypotheses of Theorem~\ref{fiber}, we have
 \[   \frac{\gr D_A}{\sqrt{\frak m_p \gr D_A}} \cong S_B \otimes \KK[\delta_1,\dots,\delta_e],
 \]
 where $S_B$ is the toric algebra generated by a simplicial scored semigroup $$\NN B \cong \frac{\NN A + \ZZ(A\cap \tau)}{\ZZ(A \cap \tau)}$$ and 
 $\KK[\delta_1,\dots,\delta_e]$ is a polynomial ring in $e (= \dim \tau)$ variables.
 \end{theorem}
 \begin{proof} First, notice that $t^a$, $a \in \NN (A \cap \tau)$, acts as a unit on $\frac{\gr D_A}{\sqrt{\frak m_p \gr D_A}}$ because $ p \in O_{\tau}$.
 So, by abusing the notation 
 \[   \frac{\gr D_A}{\sqrt{\frak m_p \gr D_A}} \cong   \frac{\gr D_A[\tau^{-1}]}{\sqrt{\frak m_p \gr D_A[\tau^{-1}]}} \] 
 where $\tau^{-1}$ means we invert $t^a$ for all $a \in \NN (A \cap \tau).$
 Note also that $$\gr D_A [\tau^{-1}] \cong \gr D(S_A[\tau^{-1}]) = \gr D(\KK[\NN A + \ZZ(A\cap \tau)]).$$
 
 Next, chose a simplicial scored semigroup $\NN B$ so that 
 $$\NN A + \ZZ(A \cap \tau) = \NN B \oplus \ZZ(A \cap \tau).$$ 
 To do this, let's first assume $\NN A$ is normal.
 By an exercise of section 1.2 in \cite{Fu}, $\frac{\RR_{\geq 0} A + \RR \tau}{\RR \tau}$ is a simplicial rational polyhedral cone with facets $\frac{\gamma + \RR \tau}{\RR
 \tau}$ where the $\gamma$'s are the facets of $\RR_{\geq 0} A$ containing $\tau$. So $\frac{\NN A + \ZZ(A\cap \tau)}{\ZZ(A \cap \tau)}$ is a simplicial normal semigroup in $\ZZ^d/ \ZZ(A \cap \tau)$. $\NN B$ can be obtained by choosing 
 suitable elements in $[\NN A + \ZZ(A \cap \tau)] \setminus \ZZ(A \cap \tau)$. For general simplicial scored semigroup $\NN A$, we just have to notice that
 $\widetilde{\NN A} \setminus \NN A$ is a union of hyperplane sections parallel to some facets of $\RR_{\geq 0} A$. So $\widetilde{\ZZ (A \cap \tau)} = \ZZ(A \cap \tau)$ and
 $\widetilde{[\NN A + \ZZ(A \cap \tau)] } \setminus [\NN A + \ZZ(A \cap \tau)]$ is a union of hyperplane sections parallel to some facets of $\RR_{\geq 0} A$ containing $\tau$.
 
 Now, $$\gr D_A [\tau^{-1}] \cong \gr D(\KK[\NN B \oplus \ZZ(A\cap \tau)]) \cong \gr D_B \otimes \gr D_{\ZZ(A \cap \tau)}.$$
 Therefore,
  \[\begin{aligned}
     \frac{\gr D_A}{\sqrt{\frak m_p \gr D_A}} &\cong \frac{\gr D_B \otimes \gr D_{\ZZ(A \cap \tau)}}{\sqrt{\frak m_p \gr D_B \otimes \gr D_{\ZZ(A \cap \tau)}}} \\
                                              & \cong \frac{\gr D_B}{\sqrt{\frak m_B \gr D_B}} \otimes \KK[\delta_1,\dots,\delta_e] \\
					      &\cong S_B \otimes \KK[\delta_1,\dots,\delta_e]
     \end{aligned}
     \]
     by Theorem~\ref{fiber}, where $\delta_1,\dots,\delta_e$ are the standard derivations of $\KK[\ZZ(A \cap \tau)]$.
 \end{proof}

 \subsection{Characteristic cycles}

 Let $D := D(R;\KK)$ as defined in section~\ref{pre}.
 Let $M$ be a $D$-module with a filtration $\{ M_i \}$ such that $D_iM_j \subseteq M_{i+j}$.
 The associated graded module $\gr M := \bigoplus M_i/M_{i-1}$ has the natural $\gr D$-module structure. 
 We call $\{ M_i\}$ a good filtration if $\gr M$ is finitely generated over $\gr D$.
 If $M$ is finitely generated over $D$ by $x_1,\dots, x_n$, then the filtration $\{ \sum^n_{j=1}D_ix_j \}$ is good.
 
 From now on, we assume that $\gr D$ is Noetherian. This is always the case when $R$ is regular.
 
 For a $D$-module $M$ with a good filtration $\{M_i\}$, define the characteristic variety $\Ch (M)$ of $M$ to be the support of the $\gr D$-module $\gr M$.
 \[ \text{i.e. } \Ch (M) = \Var (\text{ann}_{\gr D} \gr M ) \subseteq \Spec(\gr D). \]
 The characteristic cycle $\cycle (M)$ of $M$ is the formal sum of the irreducible components $V_i$ of $\Ch (M)$ counted with multiplicity. More precisely,
 \[  \cycle (M) = \displaystyle{\sum m_i V_i}
 \] where the multiplicity $m_i$ is the length of the $(\gr D)_{p_i}$-module $(\gr M)_{p_i}$ and $p_i$ is the prime ideal corresponding to $V_i$.
 
 $\Ch (M)$ and $\cycle (M)$ do not dependent on the choice of good filtration.
 A more detailed discussion about characteristic varieties can be found in \cite{Ginsburg}.
 
 \begin{example}\label{CC0}  $R$ is naturally a $D$-module generated by the identity $1$. With the filtration $\{D_i\cdot 1\}$, $\gr R = R$ is the $\gr D$-module generated by $1$.
 So $R \cong \gr D / \text{ann}_{\gr D}(1)$ and therefore $\Ch(R) = \Var(\text{ann}_{\gr D}(1))$ is abstractly isomorphic to the ambient variety $\Spec R$.
 \end{example}
 
 As we mentioned in the introduction, many invariants of $H^i_I(R)$ can be computed via the characteristic cycles when $R$ is a polynomial algebra over $\KK$
 (see e.g. \cite{Montaner1}, \cite{Montaner2}).
 In this subsection, we compute the characteristic cycles of some local cohomology modules $H^i_I(S_A)$ using our results of finiteness properties in section~\ref{finiteness}.
 By the main result in \cite{Saito-Traves04},  $\gr D_A$ is finitely generated over $\KK$ if and only if $\NN A$ is scored. 
 In particular, $\gr D_A$ is Noetherian when $\NN A$ is scored, so it makes sense to talk about characteristic cycles in this case.
 
 \begin{example}\label{CC1} For a $1$-dimensional toric algebra $S_A = \KK[t^{a_i} \mid i=1,\dots, n ]$, $\Ch(H^1_I(S_A))$ is particularly simple.
 As in Proposition~\ref{notCM}, \[ \gr D_A = \KK \left[ t \xi, t^{\mid \Omega(w) \mid} \xi^{\mid \Omega(-w) \mid} : \lvert w \rvert \in \{a_1,\dots,a_n \}
 \cup \Hole (\NN A) \right]. \]
 For any monomial ideal $I \neq 0$ of $S_A$, $H^1_I(S_A) = \frac{D_A \cdot (1/t)}{S_A}$ by Theorem~\ref{cyclic}. Notice that $t^{a_i} \in \sqrt{\text{ann}_{\gr D_A} \gr (H^1_I(S_A))}$
 and that $\xi^{a_i} \notin \sqrt{\text{ann}_{\gr D_A} \gr (H^1_I(S_A))}$. 
 So by  Theorem~\ref{fiber}, \[  \frac{\gr D_A}{\sqrt{\text{ann}_{\gr D_A}\gr (H^1_I(S_A))}} = \KK [\delta_i \mid i = 1,\dots,n] \cong S_A \]
 where $\delta_i$ is the image of $\xi^{a_i}$ in $\frac{\gr D_A}{\sqrt{\text{ann}_{\gr D_A}\gr (H^1_I(S_A))}}.$
 Therefore, $\Ch (H^1_I(S_A))$ is abstractly isomorphic to the ambient toric variety $\Spec (S_A)$.
 Furthermore, in view of the exact sequence
 \[ 0 \rightarrow S_A \rightarrow D_A\cdot(1/t) \rightarrow H^1_I(S_A) \rightarrow 0,\]
 we have  \[  \cycle(S_A[1/f]) = \Ch(S_A) + \Ch (H^1_I(S_A)) \text{ for any monomial }f \in S_A \]
 by the additivity of the characteristic cycles on exact sequences.
 In particular, $\Ch(S_A[1/f])$ has two components each of which is abstractly isomorphic to $\Spec (S_A)$.
 \end{example}

 \begin{example}\label{CC2}
 Consider the toric algebra in Example~\ref{2dim}.
 Let $p= \Theta_t$ and $q= \Theta_s$.
 Then we have
 \[  \gr D_A = \KK[t,ts^2,ts,t^{-1}(2p-q)^2, s(2p-q), t^{-1}s^{-1}(2p-q)q, t^{-1}s^{-2}q^2,s^{-1}q, p, q].
 \]
 Set
 \[ \begin{aligned}
      &a=t, b=ts^2,c=ts,d=t^{-1}(2p-q)^2, e=s(2p-q), \\ 
                &f=t^{-1}s^{-1}(2p-q)q, g=t^{-1}s^{-2}q^2,h=s^{-1}q, i=p, j=q,
	\end{aligned}
 \]
 and consider the surjection $\phi: \KK[a,\dots,j] \rightarrow \gr D_A$.
 The following table gives the information about the characteristic cycles. 
 Again, notice that each $M$ is cyclic by Theorem~\ref{cyclic}, 
 so $\ann_{\gr D_A} \gr M$  is easy to compute.
 
 \begin{tabular}{|c|c|c|}
 
 \hline
 
 $M $                 &    $ J= \phi^{-1}(\ann_{\gr D_A} \gr M)$      &   primary decomposition of $\sqrt{J}$        \\
 
 \hline
   $S_A[1/t]$         &	$(f,g,h,i,j)+ \text{Ker} \phi$                 &	  $(a,c,f,g,h,i,f,bd-e^2)\cap (d,e,f,g,h,i,j,ab-c^2)$	\\
 
 \hline
   $H^1_{(t)}(S_A)$   &  $(f,g,h,i,j,a)+ \text{Ker} \phi$      	       &	$(a,c,f,g,h,i,f,bd-e^2)$				\\
 
 \hline
   
   $S_A[1/ts^2]$      &	 $(d,e,f,i,j)+ \text{Ker} \phi$    		&	$(j, i, h, g, f, e, d, ab - c^2 )\cap (j, i, f, c, b, e, d, ag - h^2 )$	\\
 
 \hline
 
   $H^1_{(ts^2)}(S_A)$   &  $(d,e,f,i,j,b)+ \text{Ker} \phi$   	       &	$(j, i, f, e, d, c, b, ag - h^2 )$			\\
 
 \hline
 
 $S_A[1/ts]$      &	 $(i,j)+ \text{Ker} \phi$    		&	
 $\begin{aligned} &(j, i, h, c, a, b, e, dg - f^2 )\cap (j, i, h, g, f, c, a, bd - e^2 )\\
                  &\cap (j, i, h, g, f, e, d, ab - c^2 )\cap (j, i, f, c, b, e, d, ag - h^2 )  \end{aligned}$	\\
 
 \hline
 
 $H^2_{(t,ts^2)}(S_A)$   &  $(a,b,c,e,h,i,j) + \text{Ker} \phi$        &      $(a,b,c,e,h,i,j, dg-f^2)$ \\
 
 \hline
 
 \end{tabular}
 
 \end{example} 
 
  \begin{example}\label{CC3}
  Consider the toric algebra in Example~\ref{badfiber} and use the notation there.
  By Theorem~\ref{cyclic}, we see that
  $$J:=\phi^{-1}  (\text{ann}_{\gr D_A} \gr (H^1_{(s)}(S_A))) = (a,e,f,g,i,j,k,m,n) + \text{Ker} \phi$$ and the primary decomposition of $\sqrt{J}$ is
    \[ \begin{aligned}
    \sqrt{J} = &  (o, n, m, k, j,i, c, a, f, g, h, e) \cap (n, m, l, k, j, i, d, a, f, g, h, e) \\
               & \cap (n, m, k, j, i, d, c, a, f, g, e, bh - lo).   \\   
    \end{aligned} \]
    So $\cycle (H^1_{(s)}(S_A))$ is a sum of varieties which are not all isomorphic.
    
    To compute $\Ch (H^3_{\frak m}(S_A))$, we need
    $$J':=\phi^{-1}  (\text{ann}_{\gr D_A} \gr (H^3_{\frak m}(S_A))) = (a,b,c,d,e,f,g,l,m,n,o) + \text{Ker} \phi $$ 
    and$\sqrt{J'} = (a,b,c,d,e,f,g,l,m,n,o,hi-jk).$
  
 \end{example}
  
  Inspired by Examples \ref{CC1}, \ref{CC2}, \ref{CC3}, we have the following 
 
 \begin{theorem}\label{chmax} For any scored pointed semigroup $\NN A$, the characteristic variety
 $\Ch(H^d_{\frak m}(S_A))$ is abstractly isomorphic to the ambient toric variety $\Spec (S_A)$.
 \end{theorem}
 \begin{proof}
 Since $\NN A$ is pointed, by the Ishida complex \begin{equation}\label{ishida} H^d_{\frakm}(S_A)= \frac{\KK[\ZZ^d] }{ \KK [\bigcup_{\sigma: \text {facet}} (\NN A + \ZZ(A \cap \sigma))]} .\end{equation}
 By Theorem~\ref{cyclic}, $\KK[\ZZ^d] = D_A \cdot (1/t^{\alpha})$ for some interior point $\alpha$, i.e. for some 
 $\alpha \in \NN A \setminus [\bigcup_{\sigma:\text {facet}} \NN (A \cap \sigma)].$
 Consider the expression of $\gr D_A$ in Lemma~\ref{gr}.
 Denote $J:= \ann_{\gr D_A} \gr (H^d_{\frak m}(S_A))$.
 Notice that:
 \begin{enumerate}
  
  \item For any $a \in \ZZ^d \setminus (- \widetilde{\NN A})$ and for any facet $\sigma$ with $F_{\sigma}(a) >0$, 
  there exists $n \in \NN$ so that $na - \alpha \in \NN A + \ZZ(A \cap \sigma)$. Therefore, 
  $$t^a \KK[\Theta_1,\dots,\Theta_d]\cdot P_a \subseteq \sqrt{J}$$ for all $a \in \ZZ^d \setminus (- \widetilde{\NN A})$.
  
  \item For $a = 0 \in \ZZ^d$, $P_a = 1$ and $\theta_i \cdot t^{-\alpha} = -\alpha_i t^{-\alpha} \in H^d_{\frak m}(S_A) $. So by considering the order filtration,
   $\Theta_i \in J$ for all $i=1,\dots,d$.
  
  \item By exactly the same argument as in the proof of Theorem~\ref{fiber}(3), we have $t^a P_a \in \sqrt{J} $ for $a \in [-\widetilde{\NN A} \setminus (- \NN A)]$.
  Also, for $a \in - \NN A \setminus \{0\}$, $t^a P_a$ is a product of some $t^{-a_i} P_{-a_i}$.
  
  \item For $a \in - \NN A \setminus \{0\}$, $t^aP_{a} \cdot t^{-\alpha} =P(-\alpha)t^{a-\alpha} \neq 0 $ in $\gr (H^d_{\frak m}(S_A))$ because 
  $\alpha$ is an interior point and because $a- \alpha \in \ZZ^d \setminus \bigcup_{\sigma: \text {facet}} \KK [\NN A + \ZZ(A \cap \sigma)].$
  So $t^aP_{a} \notin \sqrt{J}$ for $a \in - \NN A \setminus \{0\}$.
    \end{enumerate}
   Therefore, 
   \[   \frac{\gr D_A } {\sqrt{J}} = \KK \left[ \overline{t^{-a_i}P_{-a_i}} \big| a_i :\text{ columns of }A \right]\]
   which is isomorphic to $S_A$ by a similar argument as in 
   the final part of the proof of Theorem~\ref{fiber}.
  \end{proof}
 
 \begin{remark}\label{ref2} We mention that the description of $\sqrt{J}$ in Theorem~\ref{chmax} can be obtained alternatively using Theorem~6.2(2) in \cite{S10}.
 
 For $\beta$, let $\frakm_{\beta}$ be the ideal $(\theta_1 - \beta_1 , \dots, \theta_d - \beta_d)$ of the polynomial ring $\KK [ \theta_1, \dots, \theta_d]$. Note that $\frakm_{\beta} = \frakm_0 +\beta$ in the notation of \cite{S10}. Then $L(\frakm_{\beta}) = D_A / I(\frakm_{\beta}) $ is an irreducible $D_A$-module (see Theorem 4.1.6 in \cite{Saito-Traves01}, (5.2) in \cite{S10}).\\
 Take $\alpha \in \ZZ^d$ as in the proof of Theorem~\ref{chmax}. Then we claim that 
 \begin{equation}\label{HcongL} H^d_{\frakm_A}(S_A) \cong L(\frakm_{-\alpha}).
 \end{equation}
 Since $\tau(\frakm_0 -\alpha) = -\RR_{\geq 0}A$, $\sqrt{J} = \sqrt{\gr I(\frakm_{-\alpha})} = \mathfrak{P}(\frakm_0, - \RR_{\geq 0}A)$ by \eqref{HcongL} and Theorem 6.2(2) in \cite{S10}.\\
 The claim \eqref{HcongL} can be proved as follows: By Theorem \ref{cyclic}, $H^d_{\frakm_A}(S_A) = D_A \cdot t^{-\alpha}$.
 Since $\NN A $ is scored, $\ZZ^d \cap \QQ(A \cap \tau) = \ZZ(A \cap \tau)$ for all faces $\tau$. Hence, for $\beta \in \ZZ^d$, $\beta \nsim -\alpha$
 if and only if $\beta \in \NN A + \ZZ(A \cap \tau)$ for some face $\tau$ if and only if $\beta \in \NN A + \ZZ(A \cap \sigma)$ for some facet $\sigma$. Therefore, we see that $I(\frakm_{-\alpha}) \cdot t^{-\alpha} = 0$ by \eqref{ishida}.
 Thus there exists a surjective $D_A$-homomorphism from $L(\frakm_{-\alpha})$ to $H^d_{\frakm_A}(S_A)$. 
 Since $L(\frakm_{-\alpha})$ is irreducible, it is an isomorphism.
 \end{remark}


\begin{thebibliography}{Har70}

\bibitem[{\'A}M00]{Montaner1}
Josep {\'A}lvarez~Montaner, \emph{Characteristic cycles of local cohomology
  modules of monomial ideals}, J. Pure Appl. Algebra \textbf{150} (2000),
  no.~1, 1--25. 

\bibitem[{\'A}M04]{Montaner2}
\bysame, \emph{Characteristic cycles of local cohomology modules of monomial
  ideals. {II}}, J. Pure Appl. Algebra \textbf{192} (2004), no.~1-3, 1--20.
  

\bibitem[{\'A}MBL05]{ABL}
Josep {\'A}lvarez-Montaner, Manuel Blickle, and Gennady Lyubeznik, \emph{Generators
  of {$D$}-modules in positive characteristic}, Math. Res. Lett. \textbf{12}
  (2005), no.~4, 459--473. 
  
\bibitem[Bj{\"o}79]{Bj}
J.-E. Bj{\"o}rk, \emph{Rings of differential operators}, North-Holland
  Mathematical Library, vol.~21, North-Holland Publishing Co., Amsterdam, 1979.
  
\bibitem[B{\o}g95]{Bog95}
Rikard B{\o}gvad, \emph{Some results on {$ D$}-modules on {B}orel varieties in
  characteristic {$p>0$}}, J. Algebra \textbf{173} (1995), no.~3, 638--667.
 
\bibitem[B{\o}g02]{Bog02}
\bysame, \emph{An analogue of holonomic {$ D$}-modules on smooth varieties in
  positive characteristics}, Homology Homotopy Appl. \textbf{4} (2002), no.~2,
  part 1, 83--116 (electronic), The Roos Festschrift volume, 1. 

\bibitem[BS98]{BS}
M.~P. Brodmann and R.~Y. Sharp, \emph{Local cohomology: an algebraic
  introduction with geometric applications}, Cambridge Studies in Advanced
  Mathematics, vol.~60, Cambridge University Press, Cambridge, 1998.
  
\bibitem[BW99]{BW}
Yuri Berest and George Wilson, \emph{Classification of rings of differential
  operators on affine curves}, Internat. Math. Res. Notices (1999), no.~2,
  105--109. 

\bibitem[Ful93]{Fu}
William Fulton, \emph{Introduction to toric varieties}, Annals of Mathematics
  Studies, vol. 131, Princeton University Press, Princeton, NJ, 1993, The
  William H. Roever Lectures in Geometry. 

\bibitem[Gin86]{Ginsburg}
V.~Ginsburg, \emph{Characteristic varieties and vanishing cycles}, Invent.
  Math. \textbf{84} (1986), no.~2, 327--402. 
  
\bibitem[Har70]{Ha}
Robin Hartshorne, \emph{Affine duality and cofiniteness}, Invent. Math.
  \textbf{9} (1969/1970), 145--164. 

\bibitem[HM03]{HM03}
David Helm and Ezra Miller, \emph{Bass numbers of semigroup-graded local
  cohomology}, Pacific J. Math. \textbf{209} (2003), no.~1, 41--66.
  
\bibitem[HM05]{HM05}
\bysame, \emph{Algorithms for graded injective resolutions and local cohomology
  over semigroup rings}, J. Symbolic Comput. \textbf{39} (2005), no.~3-4,
  373--395. 

\bibitem[HS93]{HS}
Craig~L. Huneke and Rodney~Y. Sharp, \emph{Bass numbers of local cohomology
  modules}, Trans. Amer. Math. Soc. \textbf{339} (1993), no.~2, 765--779.
  
\bibitem[ILL+07]{24h}
Srikanth~B. Iyengar, Graham~J. Leuschke, Anton Leykin, Claudia Miller, Ezra
  Miller, Anurag~K. Singh, and Uli Walther, \emph{Twenty-four hours of local
  cohomology}, Graduate Studies in Mathematics, vol.~87, American Mathematical
  Society, Providence, RI, 2007. 

\bibitem[Ish88]{Ishida}
Masa-Nori Ishida, \emph{The local cohomology groups of an affine semigroup
  ring}, Algebraic geometry and commutative algebra, {V}ol.\ {I}, Kinokuniya,
  Tokyo, 1988, pp.~141--153. 

\bibitem[Jon94]{Jones}
A.~G. Jones, \emph{Rings of differential operators on toric varieties}, Proc.
  Edinburgh Math. Soc. (2) \textbf{37} (1994), no.~1, 143--160. 

\bibitem[Lyu93]{Lyubeznik}
Gennady Lyubeznik, \emph{Finiteness properties of local cohomology modules (an
  application of {$D$}-modules to commutative algebra)}, Invent. Math.
  \textbf{113} (1993), no.~1, 41--55. 

\bibitem[Lyu97]{Lyu97}
\bysame, \emph{{$F$}-modules: applications to local cohomology and
  {$D$}-modules in characteristic {$p>0$}}, J. Reine Angew. Math. \textbf{491}
  (1997), 65--130. 

\bibitem[Lyu00]{Lyu00}
\bysame, \emph{Finiteness properties of local cohomology modules: a
  characteristic-free approach}, J. Pure Appl. Algebra \textbf{151} (2000),
  no.~1, 43--50. 

\bibitem[MM06]{Matusevich-Miller}
Laura~Felicia Matusevich and Ezra Miller, \emph{Combinatorics of rank jumps in
  simplicial hypergeometric systems}, Proc. Amer. Math. Soc. \textbf{134}
  (2006), no.~5, 1375--1381 (electronic). 

\bibitem[MS05]{MS}
Ezra Miller and Bernd Sturmfels, \emph{Combinatorial commutative algebra},
  Graduate Texts in Mathematics, vol. 227, Springer-Verlag, New York, 2005.
  
\bibitem[Mus87]{Mu87}
Ian~M. Musson, \emph{Rings of differential operators on invariant rings of
  tori}, Trans. Amer. Math. Soc. \textbf{303} (1987), no.~2, 805--827.
 
\bibitem[Mus94]{Musson95}
\bysame, \emph{Differential operators on toric varieties}, J. Pure Appl.
  Algebra \textbf{95} (1994), no.~3, 303--315. 

\bibitem[Sai10]{S10}
Mutsumi Saito, \emph{The spectrum of the graded ring of differential operators
  of a scored semigroup algebra}, Comm. in Algebra \textbf{38} (2010),
  829--847.

\bibitem[SS88]{SS88}
S.~P. Smith and J.~T. Stafford, \emph{Differential operators on an affine
  curve}, Proc. London Math. Soc. (3) \textbf{56} (1988), no.~2, 229--259.
  
\bibitem[SS90]{Schafer-Schenzel90}
Uwe Sch{\"a}fer and Peter Schenzel, \emph{Dualizing complexes of affine
  semigroup rings}, Trans. Amer. Math. Soc. \textbf{322} (1990), no.~2,
  561--582. 
  
\bibitem[SS04]{SS}
Anurag~K. Singh and Irena Swanson, \emph{Associated primes of local cohomology
  modules and of {F}robenius powers}, Int. Math. Res. Not. (2004), no.~33,
  1703--1733. 
  
\bibitem[ST01]{Saito-Traves01}
Mutsumi Saito and William~N. Traves, \emph{Differential algebras on semigroup
  algebras}, 207--226. 

\bibitem[ST04]{Saito-Traves04}
\bysame, \emph{Finite generation of rings of differential operators of
  semigroup algebras}, J. Algebra \textbf{278} (2004), no.~1, 76--103.
  

\bibitem[ST09]{Saito-Takahashi09}
Mutsumi Saito and Ken Takahashi, \emph{Noetherian properties of rings of
  differential operators of affine semigroup algebras}, Osaka J. Math.
  \textbf{46} (2009), no.~2, 529--556. 

\bibitem[TT08]{TT}
Shunsuke Takagi and Ryo Takahashi, \emph{{$D$}-modules over rings with finite
  {$F$}-representation type}, Math. Res. Lett. \textbf{15} (2008), no.~3,
  563--581. 

\end{thebibliography}
\end{document}